\documentclass{article}
\usepackage{amssymb,amsmath,amsthm,graphicx,kotex}

\def\qed{\hfill {\hbox{${\vcenter{\vbox{               
   \hrule height 0.4pt\hbox{\vrule width 0.4pt height 6pt
   \kern5pt\vrule width 0.4pt}\hrule height 0.4pt}}}$}}}

\def\utr{\, \underline{\triangleright}\, }
\def\otr{\, \overline{\triangleright}\, }
\def\tr{\, \triangleright\, }

\newtheorem{theorem}{Theorem}

\newtheorem{proposition}[theorem]{Proposition}

\theoremstyle{definition}
\newtheorem{example}{Example}
\newtheorem{definition}{Definition}
\newtheorem{remark}{Remark}

\date{}

\title{\Large \textbf{Quandle Cohomology Quiver Representations}}

\author{Sam Nelson\footnote{Email: Sam.Nelson@cmc.edu. Partially supported by Simons Foundation collaboration grant 702597 and the Fulbright U.S. Scholar Program, sponsored by the U.S. Department of State and the Japan-US Educational Comission (JUSEC).}}

\begin{document}
\maketitle

\begin{abstract}
We define a family of quiver representation-valued invariants of oriented
classical and virtual knots and links associated to a choice of finite quandle
$X$, abelian group $A$, set of quandle 2-cocycles $C\subset H^2_Q(X;A)$,
choice of coefficient ring $k$ and set of quandle endomorphisms 
$S\subset \mathrm{Hom}(X,X)$. From this representation we define
four new polynomial (or ``polynomial'' depending on $A$) invariants.
We generalize to the case of biquandles and compute some small but illustrative
examples.
\end{abstract}

\parbox{5.5in} {\textsc{Keywords:} Quandle quivers, quandle cohomology,
quiver representations, knot invariants, virtual knot invariants

\smallskip

 }

\section{Introduction}\label{I}

\textit{Quandles}, an algebraic structure encoding the Reidemeister moves 
of knot theory analogously to the way groups encode the symmetries of 
geometric spaces, were introduced in \cite{J} and independently in 
\cite{M}.
Quandle cohomology and quandle cocycle invariants of oriented classical and 
virtual knots and links, as well as quandle cocycle invariants of oriented
surface-links, were introduced in \cite{CJKLS} and further developed in 
\cite{CKS} and other work.

Quandle coloring quivers were introduced in \cite{CN} by the author and 
coauthor Karina Cho. Given a finite quandle $X$, a subset 
$S\subset \mathrm{Hom}(X,X)$ of the endomorphism set of $X$ determines 
a quiver-valued invariant of oriented knots and links whose
vertex set may be identified with the quandle homset 
$\mathrm{Hom}(\mathcal{Q}(L),X)$ whose cardinality is the quandle 
counting invariant $\Phi_X^{\mathbb{Z}}(L)$. 
Since quivers are categories, this construction 
categorifies the quandle counting invariant, and several new polynomial
knot and link invariants were obtained in \cite{CN} via decategorification.

In \cite{CN2} the quandle coloring quiver was generalized to quandle cocycle 
quivers to categorify the quandle 2-cocycle invariant. Since then, quivers
have been used to categorify many other enhancements of homset invariants
of knots, including quandle module quivers in \cite{IN}, biquandle bracket
quivers in \cite{FN}, psyquandle coloring quivers in \cite{CN0} and many more.

Many popular categorified knot and link invariants, e.g. Khovanov homology
in its various forms and extensions, start with \textit{quiver 
representations,} i.e.,
assignments of vector spaces or modules to the vertices and linear 
transformations to the edges of a quiver. In this paper we introduce a
quiver representation-valued invariant of oriented classical and virtual
knots and links using quandle cohomology called the \textit{quandle cohomology
quiver representation} invariant. As an application we decategorify in 
various ways to obtain 
four new infintie families of polynomial invariants. We generalize slightly 
to the case of biquandles and compute some toy examples.

The paper is organized as follows. In Section \ref{QQQ} we review the basics of 
quandles, quandle (co)homology and quandle coloring quivers. In Section 
\ref{QCQR} we introduce the new structure of quandle cohomology quivers
and compute an example to show the process. In Section \ref{B} we generalize
slightly to the case of biquandles. In Section \ref{EC} we collect some 
examples and computations, and we conclude in Section \ref{Q} with some 
questions for future research.

This paper, including all text, diagrams and computational code, was produced 
strictly by the author without the use of generative AI in any form. Its 
contents are solely the responsibility of the author and do not necessarily 
represent the official views of the Fulbright program, the Government of the 
United States, or the JUSEC.

\section{Quandles, Quandle Cohomology and Quandle Coloring Quivers}\label{QQQ}

We begin with a definition; see \cite{EN,J} for more details.
 
\begin{definition}
A \textit{quandle} is a set $X$ with a binary operation $\tr:X\times X\to X$
satisfying the axioms
\begin{itemize}
\item[(i)] For all $x\in X$ we have $x\tr x=x$,
\item[(ii)] For all $y\in X$, the map $\beta_y:X\to X$ defined by 
$\beta_y(x)=x\tr y$ is invertible, and
\item[(iii)] For all $x,y,z\in X$ we have $(x\tr y)\tr z=(x\tr z)\tr(y\tr z)$.
\end{itemize}
The element $\beta^{-1}_y(x)$ is often denoted $x\tr^{-1}y$, and indeed $X$ is 
a quandle under $\tr^{-1}$ known as the \textit{dual quandle} of $(X,\tr)$.
A quandle in which $\beta_y^{-1}=\beta_y$ for all $y\in X$ is called an
\textit{involutory quandle} or \textit{kei} (圭).
\end{definition}

\begin{example}\label{Ex1}
Standard examples of quandles include 
\begin{itemize}
\item \textit{$n$-fold conjugation quandles}: A group $G$ is a quandle under
$x\tr y=y^{-n}xy^n$ for any choice of $n\in\mathbb{Z}$,
\item \textit{Core quandles}: A group $G$ is a quandle (indeed, a kei) under 
$x\tr y=yx^{-1}y$,
\item \textit{Alexander quandles}: A module $M$ over the ring of Laurent 
polynomials $\Lambda=\mathbb{Z}[t^{\pm 1}]$ is a quandle under 
$\vec{x}\tr\vec{y}=t\vec{x}+(1-t)\vec{y}$ and
\item \textit{Knot quandles}: Given an oriented classical or virtual 
knot or link $L$ represented by a diagram $D$, the \textit{knot quandle} 
of $L$ (also called the \textit{fundamental quandle} of $L$, denoted 
$\mathcal{Q}(L)$) is given by a
presentation with a generator for each arc in $D$ and a relation of the form
$x\tr y=z$ at each crossing in $D$ as shown.
\[\includegraphics{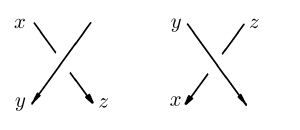}\]
\end{itemize}
\end{example}

\begin{remark}
Joyce \cite{J} shows that the knot quandle determines the knot group system
and hence determines the knot complement up to homeomorphism. For virtual knots
the knot quandle is known not to be a complete invariant \cite{K}.
\end{remark}

\begin{definition}
Let $X,Y$ be quandles. Then a map $f:X\to Y$ is a \textit{quandle homomorphism}
if for all $x,y\in X$ we have
\[f(x\tr y)=f(x)\tr f(y).\]
A quandle homomorphism $f:X\to X$ is an \textit{endomorphism}.
\end{definition}

Let $L$ be an oriented classical or virtual knot or link represented by
a diagram $D$ and let $X$ be a quandle. An assignment of an element of $X$ 
to each of the arcs in $D$ such that at every crossing we have
\[\includegraphics{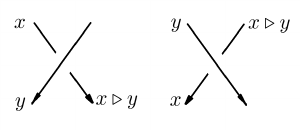}\]
(with colors passing through unchanged at any virtual crossings) is called 
an \textit{$X$-coloring} of $D$. Each $X$-coloring of $D$ determines
a unique homomorphism from the knot quandle $\mathcal{Q}(L)$ to $X$.
In particular, the quandle axioms are chosen so that this set of quandle
homomorphisms is invariant under Reidemeister moves, with the elements of the
abstract homset $\mathrm{Hom}(\mathcal{Q}(L),X)$ represented by $X$-colored
diagrams analogously to matrices encoding linear transformations and with
$X$-colored Reidemeister moves playing the role of change-of-basis matrices.

\begin{example}\label{Ex2}
Consider the (4,2)-torus link $L=L4a1$ and the quandle 
$X=\mathrm{Core}(\mathbb{Z}_4)$, given by the operation table (where we use 
the representative ``4'' for the class of zero so our elements match our 
row/column numbering starting with 1):
\[\begin{array}{r|rrrr}
\tr & 1 & 2 & 3 & 4\\ \hline
1 & 1 & 3 & 1 & 3 \\
2 & 4 & 2 & 4 & 2 \\
3 & 3 & 1 & 3 & 1\\
4 & 2 & 4 & 2 & 4
\end{array}.\]
The homset $\mathrm{Hom}(\mathcal{Q}(L4a1),\mathrm{Core}(\mathbb{Z}_4))$ 
is then 
\[
\left\{
\raisebox{-0.275in}{\scalebox{0.7}{\includegraphics{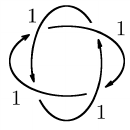}}},\
\raisebox{-0.275in}{\scalebox{0.7}{\includegraphics{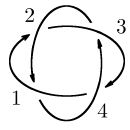}}},\
\raisebox{-0.275in}{\scalebox{0.7}{\includegraphics{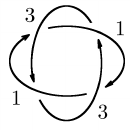}}},\
\raisebox{-0.275in}{\scalebox{0.7}{\includegraphics{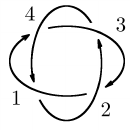}}},\
\raisebox{-0.275in}{\scalebox{0.7}{\includegraphics{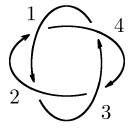}}},\
\raisebox{-0.275in}{\scalebox{0.7}{\includegraphics{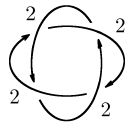}}},\
\raisebox{-0.275in}{\scalebox{0.7}{\includegraphics{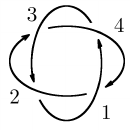}}},\
\raisebox{-0.275in}{\scalebox{0.7}{\includegraphics{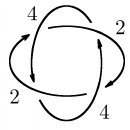}}},\
\right.\] \[\left.
\raisebox{-0.275in}{\scalebox{0.7}{\includegraphics{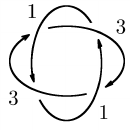}}},\
\raisebox{-0.275in}{\scalebox{0.7}{\includegraphics{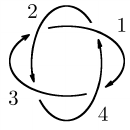}}},\
\raisebox{-0.275in}{\scalebox{0.7}{\includegraphics{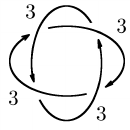}}},\
\raisebox{-0.275in}{\scalebox{0.7}{\includegraphics{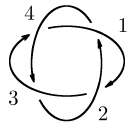}}},\
\raisebox{-0.275in}{\scalebox{0.7}{\includegraphics{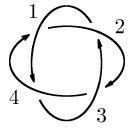}}},\
\raisebox{-0.275in}{\scalebox{0.7}{\includegraphics{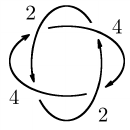}}},\
\raisebox{-0.275in}{\scalebox{0.7}{\includegraphics{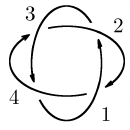}}},\
\raisebox{-0.275in}{\scalebox{0.7}{\includegraphics{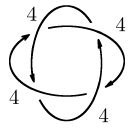}}}
\right\}.\]
\end{example}

Now, let $X$ be a finite quandle and $A$ an abelian group. Following 
\cite{CJKLS}, we recall:
\begin{itemize}
\item The module of \textit{rack $n$-chains} $C_n^R(X;A)=A[X^n]$ consists 
of all $A$-linear combinations of ordered $n$-tuples of elements of $X$,
\item The module of \textit{degenerate $n$-chains} $C_n^D(X;A)$ is the 
submodule generated by $n$-tuples of the form $(x_1,\dots,x_n)$ where 
$x_j=x_{j+1}$ for some $j$,
\item The module of \textit{quandle $n$-chains} 
$C_n^Q(X;A)=C_n^R(X;A)/C_n^D(X;A)$ is the quotient, 
also identifiable as the submodule of $C_n^R(X;A)$ with basis consisting of 
$n$-tuples without repeated neighboring entries,
\item The \textit{quandle boundary map} $\partial_n:C_n^Q(X;A)\to C_{n-1}^Q(X;A)$
is defined on generators by
\[\partial_n(x_1,\dots,x_n)=\sum_{k=1}^n 
(-1)^k((x_1,\dots, x_{k-1},x_{k+1},\dots, x_n)
-(x_1\tr x_k,\dots, x_{k-1}\tr x_k,x_{k+1},\dots, x_n))\]
and extended linearly (modulo degenerate chains),
\item The \textit{nth quandle homology} 
$H_n^Q(X;A)=\mathrm{Ker}\ \partial_n^Q/\mathrm{Im}\ \partial_{n+1}^Q$ is the 
quotient of the kernel of $\partial_n$ modulo the image of $\partial_{n+1}$, and
\item \textit{Quandle cohomology} is obtained from quandle homology by 
dualizing in the usual way. See \cite{CJKLS,EN} for more details.
\end{itemize}

An element of a quandle homset $\mathrm{Hom}(\mathcal{Q}(L),X)$ determines
an element $\vec{v}\in C_2^Q(X;\mathbb{Z})$ by interpreting each $X$-colored 
crossing as a signed ordered 2-tuple:
\[\includegraphics{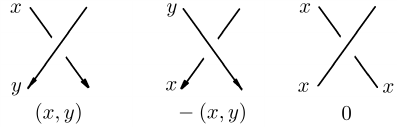}.\]

Recall that given a basis $\{\vec{v}_1,\dots,\vec{v_n}\}$ for a free $k$-module
$V$, we can represent elements of $V$ as $n\times 1$ column vectors
\[\alpha_1\vec{v}_1+\dots+\alpha_n\vec{v}_n 
= \left[\begin{array}{c} \alpha_1 \\ \vdots \\ \alpha_n\end{array}\right]\]
and that to save space we can write this column vector as the transpose of
a row vector
\[\left[\begin{array}{c} \alpha_1 \\ \vdots \\ \alpha_n\end{array}\right]
=\left[\begin{array}{ccc} \alpha_1 & \dots & \alpha_n\end{array}\right]^T.\]
Recall also that the dual $\vec{v}^*$ of a vector $\vec{v}\in V$ is a
linear functional $\vec{v}^*\in \mathrm{Hom}(V,k)$ and that it can be expressed
as the transpose of the column vector $\vec{v}$, i.e., a row vector, with
evaluation on any column vector $\vec{w}\in V$ given by the matrix product
\[\vec{v}^*(\vec{w})=\langle \vec{v}^*, \vec{w}\rangle=\vec{v}^T\vec{w}.\]

\begin{example}\label{Ex3}
Using the ordered basis 
\[B=\{(1,2),(1,3),(1,4),(2,1),(2,3),(2,4),(3,1),(3,2),(3,4),(4,1),(4,2),(4,3)\}\] 
for $C_2^Q(\mathrm{Core}(\mathbb{Z}_4);\mathbb{Z})$, the homset in Example 
\ref{Ex2} determines the multiset
\[\{
4\times[ 0 \ 0 \ 0 \ 0 \ 0 \ 0 \ 0 \ 0 \ 0 \ 0 \ 0 \ 0 ]^T,\ 
4\times[ 1 \ 0 \ 0 \ 1 \ 0 \ 0 \ 0 \ 0 \ 1 \ 0 \ 0 \ 1 ]^T,\ 
4\times[ 0 \ 0 \ 1 \ 0 \ 1 \ 0 \ 0 \ 1 \ 0 \ 1 \ 0 \ 0 ]^T, \]\[
2\times[ 0 \ 2 \ 0 \ 0 \ 0 \ 0 \ 2 \ 0 \ 0 \ 0 \ 0 \ 0 ]^T,\
2\times[ 0 \ 0 \ 0 \ 0 \ 0 \ 2 \ 0 \ 0 \ 0 \ 0 \ 2 \ 0 ]^T 
\}.\]
\end{example}

We can evaluate a quandle 2-cochain $\phi\in C^2_Q(X;A)$ on a vector 
$\vec{v}\in C_2^Q(X;\mathbb{Z})$ to obtain an element 
$\langle \phi,\vec{v}\rangle\in A$ since $A$ is a $\mathbb{Z}$-module, i.e.,
for $a\in A$ and $n\in\mathbb{Z}$ we have
\[an=\left\{\begin{array}{cl}
\overbrace{a+a+\dots+a}^n & n>0 \\
0 & n=0 \\
\overbrace{-a-a-\dots-a}^{-n} & n<0 \end{array}\right..
 \]
The quandle boundary map is chosen so that evaluating a cocycle 
$\phi\in C^2_Q(X;A)$ on a vector $\vec{v}\in C_2^Q(X;\mathbb{Z})$ arising
from a quandle homset element yields a scalar $\phi(\vec{v})\in A$ called
a \textit{Boltzmann weight} which is unchanged by $X$-colored Reidemeister
moves. More precisely, we have the following standard result (see \cite{CJKLS})
which we include for completeness:

\begin{theorem}(Carter, Jelovsky, Kamada, Langford and Saito)
If $\vec{v},\vec{v'}\in C_2^Q(X;\mathbb{Z})$ represent $X$-colored diagrams
differing by a Reidemeister move and $\phi\in H^2_Q(X;A)$, then 
$\langle \phi,\vec{v}\rangle=\langle \phi,\vec{v'}\rangle$.
\end{theorem}

\begin{proof}
Checking each of the Reidemeister moves, we see that the Boltzmann weight 
contribution is the same on both sides of the move. Here we will show one 
oriented case of each of the three Reidemeister moves; the reader is invited 
to verify the others.
\[\begin{array}{c}
\includegraphics{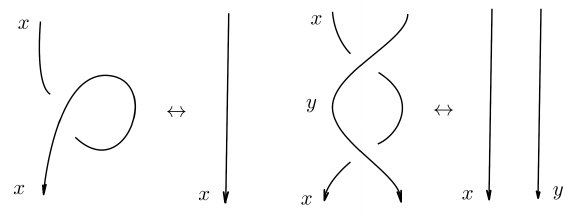} \\
\begin{array}{cc}
\phi(x,x)=0 \hspace{0.5in} 
& \hspace{0.5in} \phi(x,y)-\phi(x,y)=0 \end{array}\end{array}\]
\[\begin{array}{c}
\includegraphics{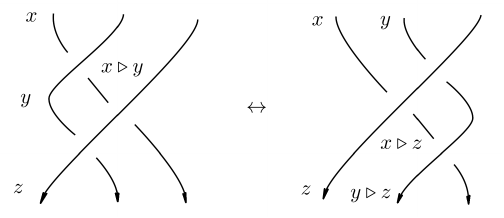}\\
\phi(x,y)+\phi(y,z)+\phi(x\tr y,z) = \phi(x,z)+\phi(y,z)+\phi(x\tr z, y\tr z) \\
\iff \phi(y,z)-\phi(y,z)-\phi(x,z)+\phi(x\tr y,z)+\phi(x,y)-\phi(x\tr z, y\tr z) =0\\
\end{array}\]
which is satisfied since $\phi\in H^2_Q(X;A)$.
\end{proof}

It follows that the multiset $M$ of Boltzmann weights for a choice 
of quandle 2-cocycle $\phi\in H^2_Q(X;A)$ is an invariant of oriented 
knots and links. For ease of comparison, it is standard to convert the 
multiset $M$ into a ``polynomial'' form 
\[\Phi_{\phi}(q)=\sum_{m\in M} q^m\]
generally known as the \textit{quandle 2-cocycle polynomial}. Evaluating 
$q=1$ then yields the quandle counting invariant 
$|\mathrm{Hom}(\mathcal{Q}(L),X)|$.
See \cite{CJKLS,EN} for more.

\begin{remark}
When $A=\mathbb{Z}$, $\Phi_{\phi}$ is a Laurent polynomial, while for other 
coefficient groups since the exponents are elements of $A$, the polynomial 
form should not be taken too literally as a polynomial but regarded as a 
convenient bookkeeping device for representing the multiset. For those 
distressed by this notation, we could instead use a complex or matrix 
representation for $A$, e.g., write $A=\mathbb{Z}_n$ multiplicatively as 
the $n$th roots of unity; alternatively, we could polynomialize $M$ by taking
$\Phi_{\phi}'(q)=\prod_{m\in M} (q-m)\in \mathbb{Z}[A,q]$ to obtain a polynomial
with $\mathbb{Z}[A]$ coefficients whose roots are the elements of $M$ and 
whose degree recovers the counting invariant.
\end{remark}

\begin{example}\label{Ex4}
In the dual of the basis for $C_2^Q(\mathrm{Core}(\mathbb{Z}_4);\mathbb{Z})$ 
from Example \ref{Ex3}, 
$H^2_Q(\mathrm{Core}(\mathbb{Z}_4);\mathbb{Z})$ has cocycles with 
representatives including
\[\phi=
[ 1 \ 0 \ 1 \ 0 \ 0 \ 0 \ 0 \ 0 \ 0 \ 0 \ 0 \ 0 ].\]
Applying $\phi$ to the homset in Example \ref{Ex3} yields
quandle 2-cocycle invariant
\[M=\{8\times 0, 8\times 1\}\]
or in polynomial form, $8+8q.$
\end{example}

Let $X$ be a finite quandle and $\sigma\in\mathrm{Hom}(X,X)$ a quandle
endomorphism. As observed in \cite{CN}, applying $\sigma$ to the colors
on an $X$-colored oriented classical or virtual knot or link diagram 
results in another $X$-colored diagram:
\[\includegraphics{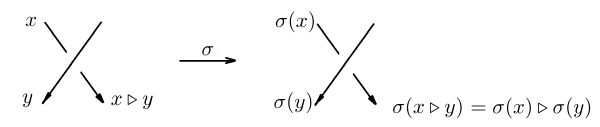}.\]

It then follows that assigning a vertex to each $X$-coloring in the 
quandle homset, each 
endomorphism determines a directed edge from each vertex to some other 
(possibly the same) vertex, resulting in a \textit{quiver} or directed 
graph called the \textit{quandle coloring quiver} associated to the 
knot or link $L$, denoted $QCQ_{X,S}(L)$. If $S=\mathrm{Hom}(X,X)$ we say
$QCQ_{X,S}(L)$ is the \textit{full quiver}.

\begin{example}
Let $X=\mathrm{Core}(\mathbb{Z}_3)$. Then writing endomorphisms
$\sigma:X\to X$ as vectors of images of elements of $X$, i.e.,
$\sigma=[\sigma(1),\sigma(2),\sigma(3)]$, 
choosing the set $S=\{[1,1,1],[2,3,1]\}\subset \mathrm{Hom}(X,X)$ we obtain
quandle coloring quiver for the trefoil knot
\[QCQ_{X,S}(3_1)=\raisebox{-0.75in}{\includegraphics{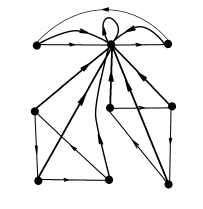}}.\]
\end{example}

\begin{theorem} (Cho., N. \cite{CN})
The quandle coloring quiver determined by a choice of quandle $X$ and
set $S\subseteq \mathrm{Hom}(X,X)$ of endomorphisms is an invariant
of oriented classical and virtual knots and links.
\end{theorem}

\begin{proof}
This follows immediately from the fact that each $\sigma$ is an endomorphism
of $X$; changing the diagram by Reidemeister moves carries the quandle 
coloring along in a bijective way, and applying the same endomorphism $\sigma$
to a diagram before or after a move commutes with the move. For example, let 
us illustrate the case of the commutative square of a Reidemeister III move:
\[\scalebox{0.80}{\includegraphics{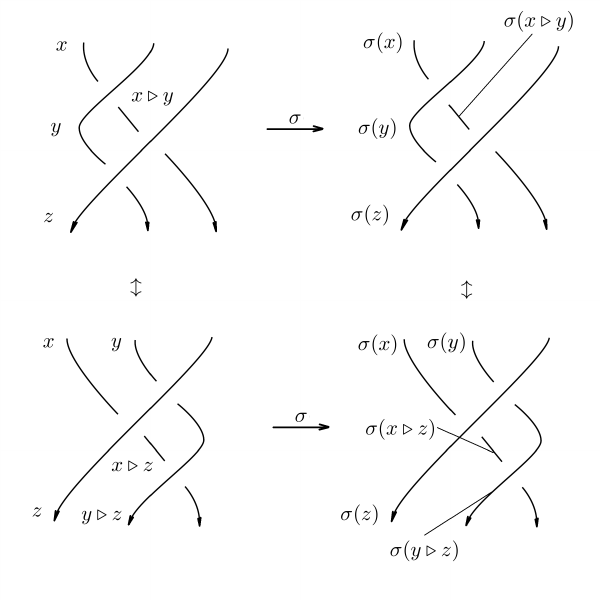}}.\]
The other moves are similar; see \cite{CN} for more details.
\end{proof}

\section{Quandle Cohomology Quivers and Representations}\label{QCQR}

In this section we introduce a new invariant of oriented classical knots
and links in the form of a quiver representation and use it to obtain
new enhancements of the quandle counting invariant.


%

Let us define the key elements of our quiver representation. Associated to
each vertex $v$ in the quiver we will have a copy of $k[A]$ with a 
distinguished subspace $V_C(\vec{v})$ and associated to each arrow from $v$ 
to $v'$ we will have a linear transformation $f_{\sigma}^{v,v'}:k[A]\to k[A]$ 
mapping the distinguished subspace at $v$ to the distinguished subspace at 
$v'$. More precsiely, we have:

\begin{definition}
Let $X$ be a finite quandle, $A$ an abelian group, $C$ a finite subset of 
$H^2_Q(X;A)$, $k$ a coefficient ring and $S$ a subset of the 
endomorphism set $\mathrm{Hom}(X,X)$.
Let $D$ be an oriented classical or virtual knot or link diagram. Then 
for each vertex $v$ representing an $X$-coloring of $D$ in 
$\mathrm{Hom}(\mathcal{Q}(D),X)$, the associated vector 
$\vec{v}\in C_2^Q(X;\mathbb{Z})$ pairs with each element $\phi\in C$ to yield
an element $\phi(\vec{v})\in A$. We then define $V_C(\vec{v})$ to be the 
subspace of $k[A]$ generated by 
$\{\phi(\vec{v})\ |\ \phi\in C\}\subset A$.
Moreover, for each arrow from a vertex $v$ to the vertex $v'=\sigma(v)$ 
induced by a permutation $\sigma$, we define a linear transformation 
$f_{\sigma}^{v,v'}:k[A]\to k[A]$ by first setting for each $\phi\in C$
\[f_{\sigma,\phi}^{v,v'}(\phi(\vec{v}))=\phi(\sigma(\vec{v}))\]
and extending $k$-linearly, then setting
\[f_{\sigma}^{v,v'}=\sum_{\phi\in C} f_{\sigma,\phi}^{v,v'}.\]
We will write $f_{\sigma}^{v,v'}$ as an $|A|\times |A|$ matrix with entries 
in $k$. 
\end{definition}

That is, each vertex $v$ gets a copy of $k[A]$ with a distinguished subspace 
$V_C(\vec{v})$ determined by the cocycle values on the coloring vector at 
the vertex for our choice of set $C$ of 2-cocycles (which may or may not be 
linearly independent). Each endomorphism $\sigma\in S$ then determines 
\textit{for each edge in the quiver} a linear transformation 
$f_{\sigma}^{v,v'}:k[A]\to k[A]$ sending distinguished subspaces 
to distinguished subspaces by sending cocycle values in $A$ considered as 
basis vectors of the copy of $k[A]$ associated to the vertex $v$ to 
cocycle values in $A$ considered as basis vectors of the copy of $k[A]$
associated to the vertex $v'=\sigma(v)$.

\begin{example}\label{Ex6}
Let $k=\mathbb{Z}$ and $A=\mathbb{Z}_3=\{0,1,2\}$  (which we might write 
as $\{\vec{0},\vec{1},\vec{2}\}$ since we are thinking of the elements of $A$ 
as generating vectors for $k[A]$).
Let $X$ be the quandle defined by the operation table
\[\begin{array}{r|rrr}
\tr & 1 & 2 & 3 \\ \hline
1 & 1 & 1 & 2 \\
2 & 2 & 2 & 1 \\
3 & 3 & 3 & 3
\end{array};\]
then $H^2_Q(X;\mathbb{Z}_3)$ includes the covectors
\[C=\{[0\ 1\ 0\ 1\ 0\ 0],\ [0\ 1\ 0\ 0\ 0\ 0],\ [0\ 0\ 0\ 0\ 1\ 1]\}\]
written in the basis $\{(1,2),(1,3),(2,1),(2,3),(3,1),(3,2)\}$ for 
$C^2_Q(X;\mathbb{Z}_3)$. The quandle coloring quiver for the link 
$L4a1$ with endomorphism $\sigma=[2\ 2\ 1]$ includes the arrow
\[\includegraphics{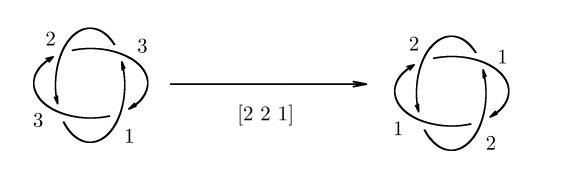}\]
from the $X$-coloring $v$ given by coloring vector 
$\vec{v}=[0\ 1\ 0\ 1\ 1\ 1]^T$ to
the $X$-coloring $v'$ given by coloring vector 
$\sigma(\vec{v})=[2\ 0\ 2\ 0\ 0\ 0]^T$. Then applying the cocycles in $C$ 
to these coloring vectors yields linear maps
\[[0\ 1\ 0\ 1\ 0\ 0]\left[\begin{array}{r} 0\\1\\0\\1\\1\\1\end{array}\right]
\stackrel{\sigma}{\longrightarrow}
[0\ 1\ 0\ 1\ 0\ 0]\left[\begin{array}{r} 2\\0\\2\\0\\0\\0\end{array}\right]
\quad \Rightarrow \quad \vec{2}\mapsto\vec{0} \quad \Rightarrow 
f_{\sigma,\phi_1}^{v,v'}=
\left[\begin{array}{rrr}
0 & 0 & 1 \\
0 & 0 & 0 \\
0 & 0 & 0 \\
\end{array}\right],
\]
\[[0\ 1\ 0\ 0\ 0\ 0]\left[\begin{array}{r} 0\\1\\0\\1\\1\\1\end{array}\right]
\stackrel{\sigma}{\longrightarrow}
[0\ 1\ 0\ 0\ 0\ 0]\left[\begin{array}{r} 2\\0\\2\\0\\0\\0\end{array}\right]
\quad \Rightarrow \quad \vec{1}\mapsto\vec{0} \quad \Rightarrow 
f_{\sigma,\phi_2}^{v,v'}=
\left[\begin{array}{rrr}
0 & 1 & 0 \\
0 & 0 & 0 \\
0 & 0 & 0 \\
\end{array}\right]
\] and
\[[0\ 0\ 0\ 0\ 1\ 1]\left[\begin{array}{r} 0\\1\\0\\1\\1\\1\end{array}\right]
\stackrel{\sigma}{\longrightarrow}
[0\ 0\ 0\ 0\ 1\ 1]\left[\begin{array}{r} 2\\0\\2\\0\\0\\0\end{array}\right]
\quad \Rightarrow \quad \vec{2}\mapsto\vec{0} \quad \Rightarrow 
f_{\sigma,\phi_3}^{v,v'}=
\left[\begin{array}{rrr}
0 & 0 & 1 \\
0 & 0 & 0 \\
0 & 0 & 0 \\
\end{array}\right].
\]
Thus we have 
$V_{C}(\vec{v})=\mathbb{Z}[\vec{1},\vec{2}]\subset \mathbb{Z}[\mathbb{Z}_3]$,
$V_{C}(\sigma(\vec{v}))=\mathbb{Z}[\vec{0}]\subset \mathbb{Z}[\mathbb{Z}_3]$ and
$f_{\sigma}^{v,v'}:\mathbb{Z}[\mathbb{Z}_3]\to \mathbb{Z}[\mathbb{Z}_3]$ defined 
by the matrix
\[f_{\sigma}^{v,v'}=f_{\sigma,\phi_1}^{v,v'}=f_{\sigma,\phi_3}^{v,v'}+f_{\sigma,\phi_3}^{v,v'}=
\left[\begin{array}{rrr}
0 & 1 & 2 \\
0 & 0 & 0 \\
0 & 0 & 0 \\
\end{array}\right].\]
\end{example}

We can now make our main new definition and state our main proposition.

\begin{definition}
Let $X$ be a finite quandle, 
$A$ an abelian group, 
$C\subset H^2_Q(X;A)$ a finite set of quandle $2$-cocycles representing 
quandle cohomology classes,
$L$ an oriented classical or virtual knot or link represented by a diagram $D$,
$k$ a coefficient ring and 
$S\subset \mathrm{Hom}(X,X)$ a set of quandle endomorphisms. 
We will refer to the 5-tuple $(X,A,C,k,S)$ as the \textit{quandle cohomology 
quiver representation data vector} or just the \textit{data vector}. 
Then we define the \textit{quandle cohomology quiver representation} of $L$ 
to be the quandle coloring quiver of $L$ with respect to $(X,S)$ with each 
vertex $\vec{v}$ weighted with the pair $(k[A],V_C(\vec{v}))$ and each edge 
defined by $\sigma\in S$ weighted with $f_{\sigma}^{v,v'}$.
\end{definition}

Then we have our main proposition:

\begin{proposition}\label{prop1}
The quandle cohomology quiver representation of an oriented classical or 
virtual knot or link is unchanged by Reidemeister moves on its diagram 
and hence is an invariant of oriented classical and virtual knots and links. 
\end{proposition}

\begin{proof}
Reidemeister moves on $D$ induce isomorphisms on the quandle coloring
quiver and do not change the cocycle values $\langle \phi,\vec{v}\rangle$. 
It then follows that vector spaces $(k[A],V_C(\vec{v}))$ and edge maps
$f_{\sigma}^{v,v'}$ are also preserved by the quiver isomorphism.
\end{proof}

\begin{example}\label{Ex7}
Computing the rest of the quiver representation for the link in Example 
\ref{Ex6}, we get 
\[\includegraphics{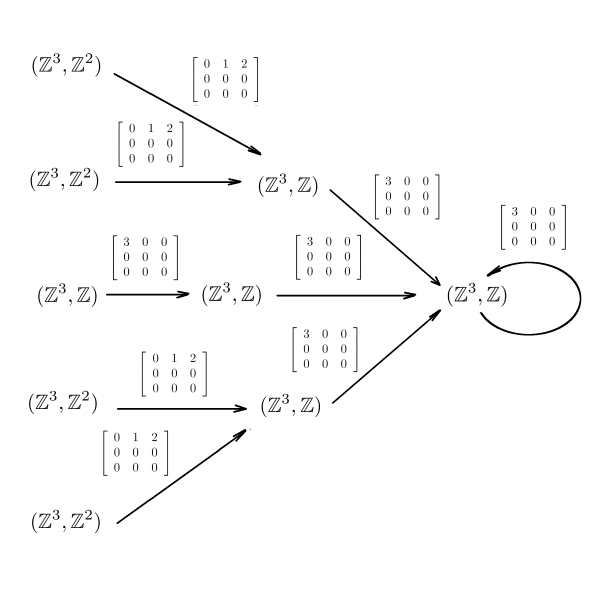}.\]

\end{example}

Quiver representations form the starting point for many knot and link
invariants such as Khovanov homology as well as many other structures
of interest. We anticipate exploring many such avenues in future work.
As an application we will now define four new infinite families of 
polynomial invaraints as decategorifications.

When defining new knot invariants, it is often necessary to strike a
compromise between strength and ease of use. One reasonable choice for 
a polynomial invariant of
matrices is the characteristic polynomial $\chi(M)=\mathrm{det}(tI-M)$; 
while not completely determining $M$, it nonetheless retains 
most of the information about the linear transformation and is generally 
much more 
compact and easy to compare. Another convenient choice is to encode 
a square matrix $M$ as a \textit{matrix polynomial}
\[p_M(x,y)=\sum_{j,k\in A} M_{j,k}x^jy^k\]
where $A$ is a set of labels for the rows and columns of $M$. 
This polynomial actually determines the matrix, given the
row/column labels:
\[\left[\begin{array}{rrr}
1 & 0 & 2 \\
3 & 0 & 0 \\
0 & 0 & 2
\end{array}\right]
\leftrightarrow 
1x^0y^0+2x^0y^2+3x^1y^0+2x^2y^2
\]
assuming a label set of $A=\{0,1,2\}$. We caution
that depending on the choice of $A$, we should regard this ``polynomial''
as a bookkeeping device and not take it too literally as a polynomial
since the exponents are elements of $A$.

With these in mind, let us now define our new polynomial invariants.

\begin{definition}
Let $Q$ be a quandle cocycle quiver representation with
data vector $(X,A,C,k,S)$. We define the 
\textit{quandle cohomology quiver representation edge characteristic 
polyonomial} or just the \textit{edge characteristic polynomial} to be the sum
\[\Phi_{\chi}^E(t)=\sum_{e\in Q} \chi(f_{\sigma})\]
of the characteristic polynomials of the matrices $f_{\sigma}$
associated to each edge. We further define the
\textit{quandle cohomology quiver representation edge matrix
polyonomial} or just the \textit{edge matrix polynomial} to be the sum
\[\Phi_{p_M}^{E}(x,y)=\sum_{e\in Q} p_M(f_{\sigma})\]
of the matrix polynomials encoding the matrices representing $f_{\sigma}$
on each edge.
\end{definition}

\begin{example}
The quiver representation in Example \ref{Ex7} has edge characteristic 
polynomial $\Phi_{\chi}^E(t)=9t^3-15t^2$ and edge 2-variable polynomial
$\Phi_{p_M}^E(x,y)=8y^2+4y+15$.
\end{example}

These two polynomials include information about all of the matrices 
associated to the edges in $Q$, but they don't make explicit use of 
the overall quiver structure of $Q$. For this, let us consider the
set of paths in our quiver which are:
\begin{itemize}
\item maximal in the sense of not being a subpath of any other path, and
\item non-repeating, i.e. not repeating any edge.
\end{itemize}

\begin{remark}
We note that since every vertex has out-degree $|S|$, maximal paths
may start at a leaf but cannot end at one. Generically, a maximal 
non-repeating path starts at a leaf or cycle and connects to other cycles.
\end{remark}

\begin{definition}
Let $Q$ be a quandle cocycle quiver representation. For each maximal 
directed path $P_j=e_1e_2\dots e_k$ in $Q$ without repeated edges, in
the corresponding sequence of vector spaces and maps
\[V_1\stackrel{f_1}{\to}V_2\stackrel{f_2}{\to}\dots\stackrel{f_{k}}{\to} V_{k+1},
\]
let $M_j=f_{k}\dots f_2f_1$ be the the matrix product. Then we define the
\textit{quandle cohomology quiver representation maximal path 
characteristic polynomial} or just the \textit{path characteristic polynomial}
to be the sum 
\[\Phi_{\chi}^{P}=\sum_{j} \chi(M_j)s^{|j|}\]
of characteristic polynomials of $M_j$ times a variable 
$s$ to the power of the length of the path $j$ over the set of maximal
nonrepeating paths in $Q$, and we define the
\textit{quandle cohomology quiver representation maximal path 
matrix polynomial} or just the \textit{path matrix polynomial}
to be the sum
\[\Phi_{p_M}^P(x,y)=\sum_j p_M(M_j)z^{|j|}\] of products of $p_M$ values of $M_j$
times $z$ to the power of the length of the paths $j$ over the set of 
maximal nonrepeating paths in $Q$.
\end{definition}

We then have our main application:
\begin{proposition}
The polynomials $\Phi_{\chi}^E,\ \Phi_{\chi}^P,\Phi_{p_M}^E$ and $\Phi_{p_M}^P$ 
are invariants of classical and virtual oriented knots
and links for every choice of finite quandle $X$, abelian group $A$, 
set of quandle cocycles $C\subset H^2_Q(x;A)$, coefficient ring $k$ 
and set of quandle endomorphisms $S\subset \mathrm{Hom}(X,X)$.
\end{proposition}

\begin{proof}
Changing the diagram by a Reidemeister move induces a quiver isomorphism
preserving the vector spaces $(k[A],V_C(\vec{v}))$ and the matrices 
$f_{\sigma}$. Hence, $\Phi_{\chi}^E,\ \Phi_{\chi}^P,\Phi_{p_M}^E$ and $\Phi_{p_M}^P$ 
are unchanged by Reidemeister moves and thus are knot invariants.
\end{proof}

\begin{example}
In the quiver representation in Example \ref{Ex7}, there are five maximal
paths without repeating edges, each starting from a leaf and ending at the 
loop. Then four of the edges have matrix product
\[
\left[\begin{array}{rrr}
3 & 0 & 0 \\
0 & 0 & 0 \\
0 & 0 & 0
\end{array}\right]
\left[\begin{array}{rrr}
3 & 0 & 0 \\
0 & 0 & 0 \\
0 & 0 & 0
\end{array}\right]
\left[\begin{array}{rrr}
0 & 1 & 2 \\
0 & 0 & 0 \\
0 & 0 & 0
\end{array}\right]
=
\left[\begin{array}{rrr}
9 & 0 & 0 \\
0 & 0 & 0 \\
0 & 0 & 0
\end{array}\right]
\left[\begin{array}{rrr}
0 & 1 & 2 \\
0 & 0 & 0 \\
0 & 0 & 0
\end{array}\right]
=
\left[\begin{array}{rrr}
0 & 9 & 18 \\
0 & 0 & 0 \\
0 & 0 & 0
\end{array}\right]
\]
and the other has matrix
\[
\left[\begin{array}{rrr}
3 & 0 & 0 \\
0 & 0 & 0 \\
0 & 0 & 0
\end{array}\right]^3
=\left[\begin{array}{rrr}
27 & 0 & 0 \\
0 & 0 & 0 \\
0 & 0 & 0
\end{array}\right].
\]
All maximal non-repeating paths in this quiver have length 3.
Then the path characteristic polynomial is $\Phi_{\chi}^P(s,t)=5s^3t^3-27s^3t^2$
and the path matrix polynomial is $\Phi_{p_M}^P(x,y,z)=72x^2z^3+36xz^3+27z^3$.
\end{example}

\begin{example}
If the set $S$ consists only of the identity endomorphism
and the set $C$ of cocycles consists of a single cocycle $\phi$,
then each path matrix has a single 1 in the row and column corresponding to
the basis element $\phi(\vec{v})$ of $k[A]$. Then the path matrix polynomial
equals the edge polynomial after evaluating at $z=1$ and setting $xy=q$ yields
the original quandle 2-cocycle invariant; hence, this family of invariants 
includes the quandle 2-cocycle invariants as a special case. If we further 
choose the cocycle to be a coboundary and $k=\mathbb{Z}$, then (interpreting 
$(xy)^0=1$) the polynomial is an integer, more precisely the number of quandle 
colorings of $L$, i.e., the quandle counting invariant. We further note that 
if $C$ includes several quandle cocycles and $S=\{\mathrm{Id}_X\}$ then the 
common value of the edge and path polynomials with $z=1$ is a kind of 
``multi-2-cocycle'' invariant similar to the invariant described in \cite{KV}.
\end{example}

\section{Generalization to Biquandles}\label{B}

In this section we generalize the construction from the previous section
to the case of biquandles. This has the advantage of giving us more 
invariants with smaller operation tables for quicker computation as well
as stronger invariants for the case of virtual knots and links. See \cite{EN}
for more details.

\begin{definition}
A \textit{biquandle} is a set $X$ with two binary operations
$\utr, \otr: X\times X\to X$ satisfying the axioms
\begin{itemize}
\item[(i)] For all $x\in X$, $x\utr x=x\otr x$,
\item[(ii)] For all $y\in X$ the maps $\alpha_y,\beta_y:X\to X$ 
and the map $S:X\times X\to X\times X$ defined by $\alpha_y(x)=x\otr y$,
$\beta_y(x)=x\utr y$ and $S(x,y)=(y\otr x,x\utr y)$ are invertible and
\item[(iii)] For all $x,y,z\in X$ we have the \textit{exchange laws}
\[\begin{array}{rcl}
(x\utr y)\utr(z\utr y) & = & (x\utr z)\utr(y\otr z) \\
(x\utr y)\otr(z\utr y) & = & (x\otr z)\utr(y\otr z) \\
(x\otr y)\otr(z\otr y) & = & (x\otr z)\otr(y\utr z).
\end{array}\]
\end{itemize}
\end{definition}

\begin{example}
Standard examples of biquandles include
\begin{itemize}
\item \textit{Quandles} are biquandles with $x\otr y=x$ for all $x,y\in X$,
\item \textit{Alexander biquandles} are modules over 
$\mathbb{Z}[t^{\pm 1},s^{\pm 1}]$ with $x\utr y=tx+(s-t)y$ and $x\otr y=sx$,
\item Groups are biquandles under the operations $x\utr y= y^{-1}xy^{-1}$ and
$x\otr y=x^{-1}$ and
\item Oriented classical and virtual knots and links have an associated
\textit{fundamental biquandle} or \textit{knot biquandle} defined via a 
presentation from any diagram of the knot or link. See \cite{EN} for more.
\end{itemize}
\end{example}

\begin{definition}
A map $f:X\to Y$ between biquandles such that for all $x,y\in X$ we have
\[f(x\utr y)= f(x)\utr f(y) \quad \mathrm{and}\quad f(x\otr y)
= f(x)\otr f(y)\]
is a \textit{biquandle homomorphism}; a biquandle homomorphism 
$\sigma:X\to X$ is a \textit{biquandle endomorphism}.
\end{definition}

Biquandle (co)homology is defined analogously to quandle (co)homology with
the biquandle boundary map given by
\[\partial_n(x_1,\dots,x_n)=\sum_{k=1}^n 
(-1)^k((x_1,\dots, x_{k-1},x_{k+1},\dots, x_n)
-(x_1\utr x_k,\dots, x_{k-1}\utr x_k,x_{k+1}\otr x_k,\dots, x_n\otr x_k)).\]

The coloring rule for biquandle colorings of oriented link diagrams is 
(ignoring virtual crossings)
\[\includegraphics{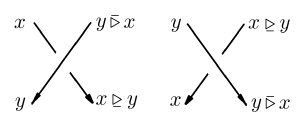}\]
where we color \textit{semiarcs} divided at both over- and under-crossings,
and translating $X$-colored diagrams to $2$-chains uses 
\[\includegraphics{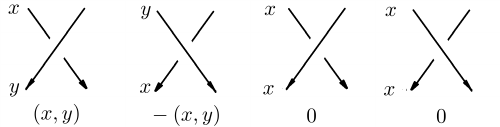}.\]

\begin{example}
The smallest nontrivial biquandle is $\mathbb{Z}_2=\{1,2\}$ (writing the 
class of zero as $2$ so our entries match our row/column numbers) with 
operations $x\utr y=x\otr y=x+1$. Then $\partial_3$ is the zero map, 
so for any $A$, the space of cocycles is the same as the space of
cochains $C^2_Q(X;A)$, generated by $\{[1\ 0],[0\ 1]\}$ when written in 
the ordered basis $\{\chi_{(1,2)},\chi_{(2,1)}\}$.
\end{example}

The quandle cocycle quiver representation and its various polynomial 
decategorifications from Section \ref{QCQR} extend to the biquandle case
by replacing quandles and quandle homomorphisms with biquandles and biquandle
homomorphisms; indeed, the former is a special case of the latter.

\begin{remark}
It was shown in \cite{J} that the knot quandle is a complete invariant of
classical knots up to ambient homeomorphism, and the knot biquandle is 
easily seen to be unchanged by reversed orientation mirror image; it follows
that for classical knots, the knot quandle already contains the same 
information as the knot biquandle. However, for virtual knots and links, the
knot biquandle is known to be stronger than the knot quandle, and in any case
the same information presented in different forms can have different levels
of accessibility -- it may be more computationally efficient to use quandles
in some cases and biquandles in others, even for classical knots and links.
\end{remark}

\section{Examples and Computations}\label{EC}

In this Section we collect some examples and computations.

\begin{example}
The link in $L4a1$ with respect to the data vector consisting of quandle 
$X=\mathrm{Core}(\mathbb{Z}_4)$, abelian group $A=\mathbb{Z}_3$, cocycles 
\[\left\{
[ 1 \ 0 \ 1 \ 0 \ 0 \ 0 \ 0 \ 0 \ 0 \ 0 \ 0 \ 0 ],\
[ 0 \ 0 \ 0 \ 0 \ 0 \ 0 \ 0 \ 1 \ 1 \ 0 \ 0 \ 0 ]
\right\}\]
in $H^2_Q(X;\mathbb{Z}_3)$, coefficient ring $k=\mathbb{Z}$ and
endomorphism set $S=\{[2,4,2,4],[1,1,1,1]\}$ has (via \texttt{python} 
computation) edge and path characteristic and matrix polynomials
\[\begin{array}{rcl}
\Phi_{\chi}^E & = & 32t^3-32t^2 \\
\Phi_{p_M}^E & = & 32y+32 \\
\Phi_{\chi}^P & = & 60s^7t^3 - 1536s^7t^2 + 22s^6t^3 - 384s^6t^2 + 16s^5t^3 + 22s^4t^3 - 96s^4t^2 \\
\Phi_{p_M}^P & = & 6144xz^7 + 1024xz^6 + 512xz^5 + 256xz^4 + 1536z^7 + 384z^6 + 96z^4
\end{array}.\]
\end{example}

\begin{example}\label{ExL}
Let $X=\mathbb{Z}_2$ with $x\utr y=x\otr y=x+1$, $A=\mathbb{Z}_3$,
$C=\{[1\ 0],[0\ 1]\}\subset H_Q^2(X;A)$ with respect to the 
ordered basis $\{(1,2),(2,1)\}$ of $C_Q^2(X;A)$, $k=\mathbb{Z}$ and
$S=\mathrm{Hom}(X,X)=\{[1,2],[2,1]\}$. Then the virtual
knot $2.1$ has (like all classical and virtual knots) two $X$-colorings,
both represented by the biquandle 2-chain $\{[1\ 1]^T\}$
resulting in biquandle coloring quiver 
\[\scalebox{0.95}{\includegraphics{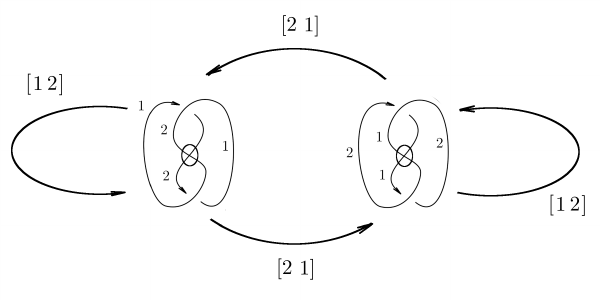}}\]
with distinguished subspace $\mathbb{Z}[\vec{1}]\subset\mathbb{Z}^3$
at each vertex and quiver representation
\[\scalebox{0.95}{\includegraphics{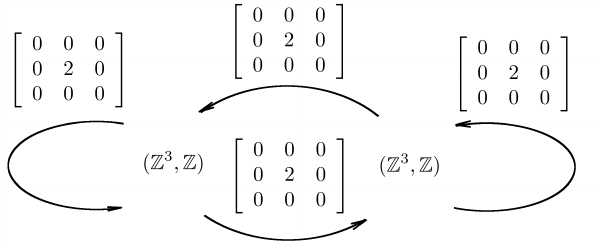}}.\]
All maximal non-repeating paths in this quiver have length 4.
We then have invariant values 
\[\begin{array}{rclrcl}
\Phi_{\chi}^E(2.1) & = & 4t^3-8t^2 & \Phi_{p_M}^E(2.1) & = & 8xy \\
\Phi_{\chi}^P(2.1) & = & 4s^4t^3-64s^4t^2& \Phi_{p_M}^P(2.1) & = & 64xyz^4 
\end{array}.\]
The mirror image of 2.1 has quiver representation
\[\scalebox{0.95}{\includegraphics{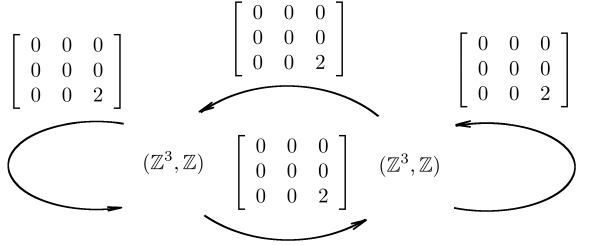}}\]
\[\begin{array}{rclrcl}
\Phi_{\chi}^E(2.1) & = & 4t^3-8t^2 & \Phi_{p_M}^E(2.1) & = & 8x^2y^2 \\
\Phi_{\chi}^P(2.1) & = & 4s^4t^3-64s^4t^2 & \Phi_{p_M}^P(2.1) & = & 64x^2y^2z^4 
\end{array}\]
showing that these invariants are sensitive to (horizontal) mirror image.
\end{example}

\begin{example}
Using \texttt{python} code, we compute the following values for the edge
characteristic and matrix polynomials for the virtual knots with up to 
4 classical crossings with a choice of orientation in the table at \cite{KA} 
using the $(X,A,C,k,S)$ data vector from Example \ref{ExL}. 
\[
\begin{array}{r|l}
\Phi_{p_M}^E(L) & L \\\hline
64z^4 & 3.1, 3.5, 3.6, 3.7, 4.2, 4.6, 4.8, 4.10, 4.12, 4.13, 4.16, 4.17, 4.19, 4.21, 4.23, 4.24, 4.26, 4.31, 4.32, 4.35, \\
 & 4.36, 4.41, 4.42, 4.46, 4.47, 4.50, 4.51, 4.55, 4.56, 4.57, 4.58, 4.59, 4.65, 4.66, 4.67, 4.68, 4.70, 4.71, \\
& 4.72, 4.75, 4.76, 4.77, 4.79, 4.85, 4.86, 4.89, 4.90, 4.93, 4.96, 4.97, 4.98, 4.99, 4.102, 4.103, 4.105, \\ 
& 4.106, 4.107, 4.108 \\
64xyz^4 & 2.1, 4.1, 4.3, 4.7, 4.25, 4.28, 4.43, 4.53, 4.73, 4.80, 4.84, 4.88, 4.91, 4.100, 4.104 \\
64x^2y^2z^4 & 3.2, 3.3, 3.4, 4.4, 4.5, 4.9, 4.11, 4.14, 4.15, 4.18, 4.20, 4.22, 4.27, 4.29, 4.30, 4.33, 4.34, 4.37, 4.38, \\
& 4.39, 4.40, 4.44, 4.45, 4.48, 4.49, 4.52, 4.54, 4.60, 4.61, 4.62, 4.63, 4.64, 4.69, 4.74, 4.78, 4.81, 4.82, \\
& 4.83, 4.87, 4.92, 4.94, 4.95, 4.101 \\
\end{array}\]
\end{example}

\begin{example}
Computing all four polynomials for a choice of orientation for the 
prime classical links with up to 7 crossings as found in \cite{KA}  
with the choice of $(X,A,C,k,S)$ data vector
from Examples \ref{Ex6}-\ref{Ex7} using \texttt{python} code, we obtain the 
tables
\[
\begin{array}{r|ll}
L & \Phi_{\chi}^E & \Phi_{p_M}^E \\ \hline
L2a1 & 5t^3 - 15t^2 & 15 \\
L4a1 & 9t^3 - 15t^2 & 8y^2 + 4y + 15 \\
L5a1 & 9t^3 - 27t^2 & 27 \\
L6a1 & 9t^3 - 15t^2 & 8y^2 + 4y + 15 \\
L6a2 & 5t^3 - 15t^2 & 15 \\
L6a3 & 5t^3 - 15t^2 & 15 \\
L6a4 & 27t^3 - 81t^2 & 81 \\
L6a5 & 15t^3 - 27t^2 & 12y^2 + 6y + 27 \\
L6n1 & 15t^3 - 27t^2 & 12y^2 + 6y + 27 \\
\end{array}\quad
\begin{array}{r|ll}
L & \Phi_{\chi}^E & \Phi_{p_M}^E \\ \hline
L7a1 & 9t^3 - 27t^2 & 27 \\
L7a2 & 9t^3 - 15t^2 & 8y^2+4y+15 \\
L7a3 & 9t^3 - 27t^2 & 27\\
L7a4 & 9t^3 - 27t^2 & 27\\
L7a5 & 5t^3 - 13t^2 & 15 \\
L7a6 & 5t^3 - 13t^2 & 15 \\
L7a7 & 15t^3 - 39t^2 & 4y^2+2y+39 \\
L7n1 & 9t^3 - 15t^2 & 4y^2+8y+15 \\
L7n2 & 9t^3 - 27t^2 & 27 
\end{array}
\] 
and
\[\begin{array}{r|ll}
L & \Phi_{\chi}^P & \Phi_{p_M}^P  \\ \hline
L2a1 & s^3t^3 - 27s^3t^2 + 2s^2t^3 - 18s^2t^2 & 27z^3 + 18z^2 \\
L4a1 & 5s^3t^3 - 27s^3t^2 & 72x^2z^3 + 36xz^3 + 27z^3 \\
L5a1 & 5s^3t^3 - 135s^3t^2 & 135z^3 \\ 
L6a1 & 5s^3t^3 - 27s^3t^2 & 72x^2z^3 + 36xz^3 + 27z^3\\
L6a2 & s^3t^3 - 27s^3t^2 + 2s^2t^3 - 18s^2t^2 & 27z^3 + 18z^2\\ 
L6a3 & s^3t^3 - 27s^3t^2 + 2s^2t^3 - 18s^2t^2 & 27z^3 + 18z^2 \\
L6a4 & 19s^3t^3 - 513s^3t^2 & 513z^3 \\
L6a5 & 7s^3t^3 - 27s^3t^2 + 3s^2t^3 - 27s^2t^2 & 108x^2z^3 + 54xz^3 + 27z^3 + 27z^2 \\ 
L6n1 & 7s^3t^3 - 27s^3t^2 + 3s^2t^3 - 27s^2t^2 & 108x^2z^3 + 54xz^3 + 27z^3 + 27z^2\\
L7a1 & 5s^3t^3 - 135s^3t^2 & 135z^3 \\ 
L7a2 & 5s^3t^3 - 27s^3t^2 & 72x^2z^3 + 36xz^3 + 27z^3\\ 
L7a3 & 5s^3t^3 - 135s^3t^2 & 135z^3\\ 
L7a4 & 5s^3t^3 - 135s^3t^2 & 135z^3\\
L7a5 & s^3t^3 - 27s^3t^2 + 2s^2t^3 - 18s^2t^2 & 27z^3 + 18z^2\\ 
L7a6 & s^3t^3 - 27s^3t^2 + 2s^2t^3 - 18s^2t^2 & 27z^3 + 18z^2\\ 
L7a7 & 7s^3t^3 - 135s^3t^2 + 3s^2t^3 - 27s^2t^2 & 36x^2z^3 + 18xz^3 + 135z^3 + 27z^2 \\
L7n1 & 5s^3t^3 - 27s^3t^2 & 36x^2z^3 + 72xz^3 + 27z^3 \\
L7n2 & 5s^3t^3 - 135s^3t^2 & 135z^3\\
\end{array}.\] 

\end{example}

\section{Questions}\label{Q}

We end with some questions and directions for future research. 

In defining knot invariants, there is often a tradeoff between strength 
and usability. We have attempted to strike a good balance by passing from
matrices to characteristic polynomials which retain much of the structure 
of the matrices while being easier to write and visually compare and with
matrix ``polynomials'' which actually encode the matrix using
row and column labels as exponents. What other decategorifications are 
possible? What are the geometric meanings of these invariants?

More generally, what kinds of new invariant structures can be built on
this foundation? Homology theories, functorial invariants, exact sequences?

We anticipate many future projects building on this and similar definitions,
including quiver representations for other homset structures, different
constructions of quiver representations and further invariants to be derived
from them.

\bibliography{sn-solo24}{}
\bibliographystyle{abbrv}

\bigskip

\noindent
\textsc{Department of Mathematical Sciences \\
Claremont McKenna College \\
850 Columbia Ave. \\
Claremont,  CA 91711}

\end{document}